\documentclass[11pt]{amsart}
\usepackage{amssymb}
\usepackage{amsfonts}
\usepackage{amsmath}
\usepackage[utf8]{inputenc}
\usepackage[mathscr]{eucal}
\usepackage{xypic}
\usepackage{graphicx}
\usepackage[english]{babel}
\usepackage{epsfig}
\usepackage{mathrsfs}
\usepackage[all]{xy}
\usepackage{tikz}
\usetikzlibrary{matrix, arrows, plotmarks}
\usepackage{pgfplots}
\def\g{{\rm{g}}}

\newtheorem{theorem}{Theorem}

\newtheorem{corollary}[theorem]{Corollary}

\newtheorem{lemma}[theorem]{Lemma}
\newtheorem{proposition}[theorem]{Proposition}
\theoremstyle{definition}

\newtheorem{example}[theorem]{Example}

\def\max{{\rm max}}
\def\min{{\rm min}}
\def\msg{{\rm msg}}

\def\m{{\rm{m}}}

\def\e{{\rm{e}}}
\def\e{{\rm{e}}}
\def\g{{\rm{g}}}
\def\n{{\rm{n}}}
\def\F{{\rm{F}}}

\def\N{{\mathbb N}}
\def\Z{{\mathbb Z}}

\def\F{{\rm F}}

\begin{document}

\title[Numerical semigroups of coated odd elements]{Numerical semigroups of coated odd elements}

\author{J. C. Rosales}
\address{Departamento de \'Algebra, Universidad de Granada, E-18071 Granada, Spain}
\email{jrosales@ugr.es}

\author{M. B. Branco}
\address{Departamento de Matemática, Universidade de Évora, 7000 Évora, Portugal}
\email{mbb@uevora.pt}

\author{M. A. Traesel}
\address{Departamento de Matemática, Instituto Federal de São Paulo, Caraguatatuba, SP, Brazil}
\email{marciotraesel@ifsp.edu.br}

\thanks{The first author was partially supported by MTM-2017-84890-P and by Junta de Andalucia group FQM-343. The second author was partially supported by CIMA -- Centro de
Investiga\c{c}\~{a}o em Matem\'{a}tica e Aplica\c{c}\~{o}es,
Universidade de \'{E}vora, project UIDB/04674/2020 (Funda\c{c}\~{a}o
para a Ci\^{e}ncia e Tecnologia). 2010 Mathematics Subject Classification: 20M14, 11D07.}

\begin{abstract}
 A numerical semigroup $S$ is coated with odd elements (Coe-semigroup), if $\left\{x-1, x+1\right\}\subseteq S$ for all odd element $x$ in $S$. In this note, we will study this kind of numerical semigroups. In particular, we are interested in the study of the Frobenius number, gender and embedding dimension of a numerical semigroup of this type.
\end{abstract}


\maketitle

\section{Introduction}\label{S1}

Let $\Z$ be the set of integers an let $\N=\left\{ z\in \Z ~|~ z\geq 0\right\}$.
A submonoid of $(\N,+)$  is a subset of $\N$ closed under addition, containing the zero element. A submonoid with finite complement in $\N$ is a numerical semigroup.

A numerical semigroup, $S$, is coated with odd elements, if $\left\{x-1, x+1\right\}\subseteq S$ for all odd element $x$ in $S$. Henceforth known as \textbf{Coe-semigroup}, which  will be the object of study in this work.

If $S$ is a numerical semigroup, then $\m(S)=\min(S\backslash\{0\})$, $\F(S)=\max(\Z\backslash S)$ and $\g(S)$ the cardinalty of $\N\backslash S$ are three importants invariants of $S$ known as \textbf{multiplicity}, \textbf{Frobenius number} and \textbf{gender} of $S$, respectively.

If $\mathcal  A$  is a nonempty subset of   $\N$, we denote by $\left\langle \mathcal A\right\rangle$ the submonoid of $(\N,+)$ generated by $\mathcal A$, that is, 
\[\langle \mathcal A\rangle=\left\{\sum_{i=1}^n \lambda_i\,a_i ~|~ n\in\N\backslash\left\{0\right\}, ~ a_i\in \mathcal A,~ \lambda_i\in \N ~\text{for all}  ~i\in\{1,\ldots, n\}  \right\}.\]

In \cite[Lemma $2.1$]{libro}, we can see that $\langle \mathcal A\rangle$ is a numerical semigroup if and only if $\gcd(\mathcal A)=1$.

If $M$ is a submonoid  of $(\N,+)$ and $M=\left\langle \mathcal A\right\rangle$ then we say that $\mathcal A$ is a system of generators of $M$. Moreover, if $M\neq \left\langle \mathcal B\right\rangle$ for all $\mathcal B \varsubsetneq \mathcal A$, then  we say that $\mathcal A$ is a minimal system of generators of $S$.  In \cite[Corollary 2.8]{libro}, we see that, every submonoid of $(\N,+)$ admits an unique minimal system of generators, which is finite. We denote by $\msg (M)$ the minimal system of generators of $M$ and its elements are called minimal generators. The cardinality of $\msg (M)$ is called the embedding dimension of $M$ and it is denoted by $\e(M)$.

The Frobenius problem (see \cite{alfonsin}) consists of finding formulas for the Frobenius number and the gender of a numerical semigroup in terms of its minimal system of generators. This problem was solved in \cite{Sylvester}, for numerical semigroups  with embedding dimension two. At present, this problem is open for embedding dimension greater than or equal to three.

In \cite{wilf} Wilf conjectured that if $S$ is a numerical semigroup  then $\e(S) \g(S)\leq (\e(S)-1) (\F(S)+1)$. This question is still widely open and it is  one of the most important problems in numerical semigroups theory. A very good source of the state of the art of this problem  is  \cite{delgado}.

Denote by $\mathscr C=\left\{S ~|~ S ~ \text{is an Coe-semigroup} \right\}$. In Section $2$, we will show that, if $S\in \mathscr C$ then  $S\cup\{\F(S), \F(S)-1\}\in \mathscr C$. This result will be used in Section 3, to order the elements of $\mathscr C$ in a rooted tree.
We will characterize the sons of a vertex in this tree and this will allow us to give an algorithm procedure to obtain recursively the elements of $\mathscr C$.

In Section $4$,  we will see  that given $k$ an odd positive integer, then  $\mathscr C(k)=\left\{S ~|~ S ~ \text{is an Coe-semigroup and} ~k\in S\right\}$ is a finite set. Moreover,   given $p$ a positive integer, $\mathscr C\big(Frob\leq p\big)=\left\{S ~|~ S ~ \text{is an Coe-semigroup with} ~ \F(S)\leq p\right\}$ and  $\mathscr C\big(gen\leq p\big)=\left\{S ~|~ S ~ \text{is an Coe-semigroup with} ~ \g(S)\leq p\right\}$ are also finite sets. Following the same line of previous section, we are going to order these three families of semigroups in a rooted tree.

An Coe-monoid is a submonoid of $(\N,+)$ which can be expressed as intersection of Coe-semigroups. It is clear that, the intersection of Coe-monoids is an Coe-monoid. This allows us to introduce the smallest Coe-monoid containing a subset $X$ of $\N$, denoted by $Coe(X)$. In Section $5$, we will show that  a submonoid $M$ of $(\N,+)$  is an  Coe-monoid if and only if either $M\subseteq \left\{2k ~|~k\in\N\right\}$ or $M$ is an Coe-semigroup. We will give an algorithm to compute $Coe(X)$ and will see that $Coe(X)$ is an Coe-semigroup if and only if $X$ contains at least an odd element.

 In \cite[Proposition $2.10$]{libro}, it is shown that if $S$ is a numerical semigroup, then $\e(S)\leq \m(S)$.
We say that $S$ has maximal embedding dimension (MED-semigroup) if $\e(S)=\m(S)$. This class of numerical semigroups has been well studied in the semigroup literature  (see for instance \cite{barucci}).  In Section $6$, we will study the Coe-semigroups which are MED-semigroups.

Finally, in Section $7$, we will study the Coe-semigroups with an unique odd minimal generator. We will show that this kind of numerical semigroups are equal to the set $\big\{T ~|~ T=2S\cup \big(\{2s+1\}+2S\big), ~ S ~\text{is a numerical semigroup  and}~ \{s,s+1\}\subseteq S\big\}$. We will give formulas for $\F(T)$, $\g(T)$ and $\e(T)$ as a function of $\F(S)$, $\g(S)$ and $\e(S)$. From these results, we will prove that if $S$ verifies the Wilf's conjecture, then $T$ also verifies the same conjecture. In addition, we will solve the Frobenius problem for Coe-semigroups  with embedding dimension three.


\section{First results }\label{S2}
Note that $\N$  is an Coe-semigroup with $\m(\N)=1$ and $\F(\N)=-1$. Hence, if $S$ is a numerical semigroup and $S\neq \N$, then we let us consider  that $\m(S)\in \N\backslash\{0,1\}$ and $\F(S)\in \N\backslash\{0\}$.
 
\begin{proposition}\label{1}
If  $S$ is an Coe-semigroup  and $S\neq \N$, then $\m(S)$ is an even integer and $\F(S)$ is an odd integer.
\end{proposition}
\begin{proof}
As $S\neq \N$, then $\m(S)\geq 2$ and $\F(S)\geq 1$. If $\m(S)$ is odd, then we get that $\m(S)-1\in S$, a contradiction. If $\F(S)$ is  even, then $\F(S)+1$ is an odd element belongs to $S$. Hence, we obtain that $\big(\F(S)+1\big)-1=\F(S)\in S$, a contradiction again.
\end{proof}


\begin{proposition}\label{4} 
Let $S$ be a numerical semigroup. The following conditions are equivalent.
\begin{enumerate}
\item $S$ is an Coe-semigroup.
\item   $\left\{x-1, x+1\right\}\subseteq S$ for all odd element $x$ in $\msg(S)$.
\end{enumerate}
\end{proposition}

\begin{proof}
$1)~ \textit{implies}~ 2)$ Trivial. 

$2)~ \textit{implies}~ 1)$ Let $s$ be an odd element in $S$. Clearly, there exists an odd element $x$ in $\msg(S)$, such that $s-x\in S$.  Then, we obtain that $s-1=x-1+s-x\in S$ and $s+1=x+1+s-x\in S$ and so $S$ is an Coe-semigroup.
\end{proof}

\begin{example}\label{5}
 By applying Proposition \ref {4}, we have that   $S=\langle 4,6,7\rangle$ is an Coe-semigroup, because $\left\{7-1, 7+1\right\}\subseteq S$.
\end{example}

The next result has immediate proof.

\begin{lemma}\label{6}
If $S$ is a numerical semigroup  such that $S\neq \N$, then $S\cup \{\F(S)\}$ and $S\cup \{\F(S)-1, \F(S)\}$ are also numerical semigroups.
\end{lemma}

Observe that, if $S$ is an Coe-semigroup such that $S\neq \N$, then  we have that  $S\cup \{\F(S)\}$ is not necessary an Coe-semigroup. In fact,  $S=\left\{0,6,10,\rightarrow\right\}$ is an Coe-semigroup, but $S\cup \{\F(S)\}= \left\{0,6,9,10,\rightarrow\right\}$ is not  an Coe-semigroup. In addition,  if $S$ is an Coe-semigroup then $\F(S)-1$ may or may not belong to $S$. Indeed,  we have that $S=\left\{0,6,10,\rightarrow\right\}$ is an Coe-semigroup such that $\F(S)-1\not\in S$ and  $T= \left\{0,6,8,10,\rightarrow\right\}$ is an Coe-semigroup such that $\F(T)-1\in T$.

\begin{lemma}\label{7}
If $S$ is an Coe-semigroup  such that $S\neq \N$, then $S\cup \{\F(S)-1, \F(S)\}$ is an Coe-semigroup.
\end{lemma}
\begin{proof}
Applying  Lemma \ref{6}, we obtain that $S\cup \{\F(S)-1, \F(S)\}$ is  a numerical semigroup. We know that, by Proposition \ref{1}, $\F(S)$ is an odd  integer. As $\{\F(S)-1, \F(S)+1\}\subseteq S\cup \{\F(S)-1, \F(S)\}$, we deduce that $S\cup \{\F(S)-1, \F(S)\}$ is an Coe-semigroup.
\end{proof}

As a consequence of Lemma \ref{7}, we have all the ingredients needed to give a recursive way of calculating a sequence of Coe-semigroups. For a given an Coe-semigroup $S$:

\begin{itemize}
\item $S_0=S$,
\item $S_{n+1}=\left\{ \begin{array}{ll}
  S_n\cup\left\{\F(S_n)-1,\F(S_n)\right\} & \hbox{if } S_n \neq \N  \\
  \N & \hbox{otherwise}.
    \end{array}\right.$
\end{itemize}

The next result  is trivial.

\begin{proposition}\label{8}
If $S$ is an Coe-semigroup, then there exists a sequence of Coe-semigroups, $S=S_0\subsetneq S_1 \subsetneq\cdots \subsetneq S_k=\N$. Furthermore, the cardinality of $S_{i+1}\backslash S_i$ is $1$ or $2$ for all $i\in\left\{0,1,\ldots,k-1\right\}$.
\end{proposition}

We will refer to the previous sequence as the chain of Coe-semigrops associated to $S$ and $k$ is the length of this chain, denoted by $l_S$.

The following result is easy to prove.

\begin{proposition}\label{9}
Let $S$ be an Coe-semigroup. The cardinality of the set $\left\{x\in \N\backslash S ~| ~x ~\text{is odd}\right\}$ is equal to $l_S$.
\end{proposition}

\begin{example}\label{10}
 Clearly $S=\langle 6,8,13,14,15,17\rangle= \left\{0,6,8,12,\rightarrow\right\}$ is an Coe-semigroup. The chain of Coe-semigrops associated to $S$ is the following: $S=S_0\subsetneq S_1=\left\{0,6,8,10\rightarrow\right\} \subsetneq  S_2=\left\{0,6,8,\rightarrow\right\}\subsetneq S_3=\left\{0,6,\rightarrow\right\}\subsetneq  S_4=\left\{0,4\rightarrow\right\}\subsetneq  S_5=\left\{0,2\rightarrow\right\}\subsetneq  S_6=\N$. The cardinality of $\left\{x\in \N\backslash S ~| ~x ~\text{is odd}\right\}=\left\{1,3,5,7,9,11\right\}$ is $6=l_S$.
\end{example}


\section{The tree of  Coe-semigroups}\label{S3}

A graph $G$ is a pair $(V ,E)$ where $V$ is a nonempty set and $E$ is a subset of $\left\{(u,v)~|~ u,v \in V, u\neq v\right\}$. The elements of $V$ and $E$ are called vertices and edges, respectively. A path of length $n$ connecting the vertices $u$ and $v$ of $G$ is a sequence of $n$ distinct edges of the form $(v_0,v_1), (v_1,v_2),\ldots, (v_{n-1},v_n)$ with $v_0 = u$ and $v_n= v$. 

A graph $G$ is a tree if there exists a vertex $r$ (known as the root of $G$) such that for every other vertex $v$ of $G$, there exists a unique path connecting $v$ and $r$. If $(u,v)$ is an edge of the tree then we say that $u$ is a son of $v$.

Our main goal in this section will be to build the tree whose vertex set is  $\mathscr C=\left\{S ~|~ S ~ \text{is an Coe-semigroup} \right\}$.

We define the graph $G\big(\mathscr C\big)$ as the graph whose vertices are elements of $\mathscr C$ and $(S, T)\in \mathscr C\times \mathscr C$ is an edge if $T=S\cup \left\{\F(S)-1,\F(S)\right\}$.  

As a consequence of Proposition \ref{8}, we have the following.

\begin{proposition}\label{11}
The graph $G\big(\mathscr C\big)$ is a tree with root equal to $\N$.  
\end{proposition}

It is clear that we can be build recursively the tree  $G\big(\mathscr C\big)$, starting in $\N$ and we connecting each vertex with its sons.
Hence,  we need to characterize the sons of an arbitrary vertex of this tree.

\begin{lemma} \cite[Lemma 1.7]{colloquium}\label{12}
Let $S$ be a numerical semigroup and $x\in S$. Then $S\backslash\{x\}$ is a numerical semigroup if and only if $x\in \msg (S)$.
\end{lemma}

The next result is easy to prove.

\begin{lemma}\label{13}
Let $S$ be a numerical semigroup  such that $S\neq \N$ and $\{x,x+1\}\subseteq S$. Then  $S\backslash\{x,x+1\}$  is a numerical semigroup if and only if $\{x,x+1\}\subseteq \msg (S)$.
\end{lemma}

\begin{lemma}\label{14}
Let $S$ be an Coe-semigroup and let $x$ be  an odd element in $\msg (S)$, such that $x>\F(S)$. Then  $S\backslash\{x\}$  is a son of $S$ in the tree  $G\big(\mathscr C\big)$.
\end{lemma}

\begin{proof}
By  Lemma \ref{12}, we have that $S\backslash\{x\}$ is a numerical with $\F(S\backslash\{x\})=x$. Since $x$ is odd then $S\backslash\{x\}$ is an Coe-semigroup and $x-1\in S$. Moreover, we can deduce that $\big(S\backslash\{x\}\cup \left\{\F(S\backslash\{x\}), \F(S\backslash\{x\})-1\right\}\big)=S\backslash\{x\}\cup\{x,x-1\}=S$ and thus $S\backslash\{x\}$  is a son of $S$ in the tree  $G\big(\mathscr C\big)$.
\end{proof}

\begin{lemma}\label{15}
Let $S$ be an Coe-semigroup and let $T$ be a son of $S$ in the  tree  $G\big(\mathscr C\big)$ such that $\g(T)=\g(S)+1$. Then, there exists an odd element $x$ in $\msg (S)$  with $x>\F(S)$ such that $T=S\backslash\{x\}$.
\end{lemma}
\begin{proof}
If  $T$ is a son of $S$ such that $\g(T)=\g(S)+1$, then we obtain that  $T\cup \{\F(T)\}=S$ and so $T=S\backslash\{\F(T)\}$. By using Lemma \ref{12}, we have that $\F(T)\in \msg(S)$ and  it is clear that $\F(S)<\F(T)$. Applying Proposition \ref{1},  we conclude that $\F(T)$ is odd.
\end{proof}

\begin{lemma}\label{16}
Let $S$ be an Coe-semigroup such that $\left\{\F(S)+1,\F(S)+2\right\}\subseteq \msg(S)$. Then $S\backslash\left\{\F(S)+1,\F(S)+2\right\}$ 
is a son of $S$ in the tree  $G\big(\mathscr C\big)$.
\end{lemma}
\begin{proof}
By Lemma \ref{13}, we have that  $T=S\backslash\left\{\F(S)+1, \F(S)+2\right\}$ is a numerical semigroup such that $\F(T)=\F(S)+2$. By applying Proposition \ref{1}, $F(S)$ is odd and so $\F(T)$ is also odd. Hence, we deduce that  $T$ is an Coe-semigroup. As $S=T\cup \left\{\F(T), \F(T)-1\right\}$, then we get that $T$ is a son of $S$.
\end{proof}

\begin{lemma}\label{17}
Let $S$ be an Coe-semigroup and let $T$ be a son of $S$ in the  tree  $G\big(\mathscr C\big)$ such that $\g(T)=\g(S)+2$. Then, $\left\{\F(S)+1,\F(S)+2\right\}\subseteq \msg(S)$ and $T=S\backslash \left\{\F(S)+1, \F(S)+2\right\}$.
\end{lemma}
\begin{proof}
If $T$ is a  son of $S$, then $T$ is an Coe-semigroup and $S=T\cup \left\{\F(T), \F(T)-1\right\}$. Since $\g(T)=\g(S)+2$, this implies that $\F(T)-1\not\in T$.  By applying Proposition \ref{1}, $F(T)$ is odd and so $\F(T)-2$ is also odd. As $\F(S)-1\not\in T$ and $T$ is an Coe-semigroup, we obtain that $\F(T)-2\not\in T$ and thus $\F(S)=\F(T)-2$. Consequently, $T=S\backslash \left\{\F(S)+1, \F(S)+2\right\}$ and by using Lemma \ref{13}, we conclude that $\left\{\F(S)+1,\F(S)+2\right\}\subseteq \msg(S)$.
\end{proof}

As a consequence of Lemmas \ref{14}, \ref{15}, \ref{16} and \ref{17} we obtain the following result.

\begin{theorem}\label{18}
Let $S$ be an Coe-semigroup,  then set of sons of $S$ in the  tree  $G\big(\mathscr C\big)$  is equal to:
\begin{enumerate}
\item  $\left\{S\backslash\{x\} ~|~ x ~ \text{is an odd element in}~ \msg (S) ~\text{with}~x>\F(S) \right\}$  if $\quad$ $\left\{\F(S)+1,\F(S)+2\right\}\nsubseteq \msg(S)$. 
\item  $\left\{S\backslash\{x\} ~|~ x ~ \text{is an odd element in}~ \msg (S) ~\text{with}~x>\F(S) \right\}\cup$ $\left\{S\backslash \left\{\F(S)+1, \F(S)+2\right\}\right\}$  ~ if  ~  $\left\{\F(S)+1,\F(S)+2\right\}\subseteq \msg(S)$. 
\end{enumerate}
\end{theorem}

It is clear that for all $k\in \N\backslash\{0\}$ we have that $\left\{0,2k, \rightarrow\right\}$ is an Coe-semigroup, this implies that the set $\mathscr C$ has infinite cardinality.

The last theorem can be used to recurrently construct the tree  $G\big(\mathscr C\big)$, starting in $\N$, containing the set of all  Coe-semigroups.

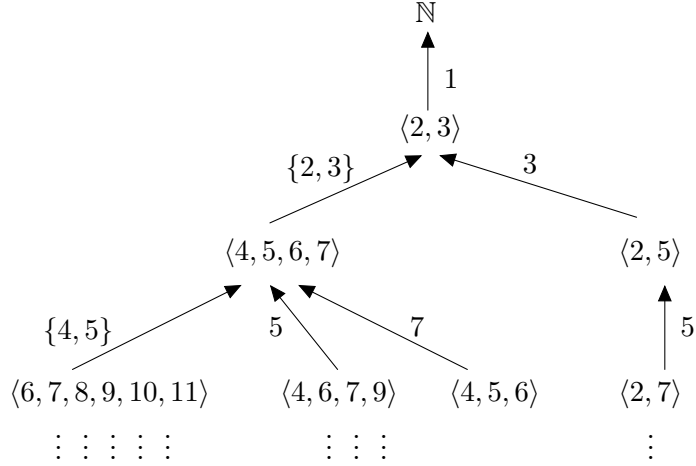
\begin{tikzpicture}[line cap=round,line join=round,>=triangle 45,xscale=0.75,yscale=0.75]
\clip(0,-5) rectangle (16,5);

\draw [<-] (10.3,3.5) -- (10.3,2.1);

\draw [<-] (10.2,1.3) -- (7.5,0.1);
\draw [<-] (10.5,1.3) -- (14,0.2);

\draw [<-] (14.5,-1) -- (14.5,-2.5);
\draw [<-] (7,-1) -- (4,-2.5);
\draw [<-] (7.5,-1) -- (8.7,-2.5);
\draw [<-] (8,-1) -- (11,-2.5);

\draw (9.9,4.2) node[anchor=north west] {$\N$};

\draw (9.6,2.2) node[anchor=north west] {$\langle 2,3\rangle$};

\draw (6.5,0) node[anchor=north west] {$\langle 4,5,6,7\rangle$};
\draw (13.5,0) node[anchor=north west] {$\langle 2,5\rangle$};

\draw (2.7,-2.5) node[anchor=north west] {$\langle 6,7,8,9,10,11\rangle$};
\draw (7.5,-2.5) node[anchor=north west] {$\langle 4,6,7,9\rangle$};
\draw (10.5,-2.5) node[anchor=north west] {$\langle 4,5,6\rangle$};
\draw (13.5,-2.5) node[anchor=north west] {$\langle 2,7\rangle$};

\draw (3.5,-3.2) node[anchor=north west] {$\vdots~~\vdots~~\vdots~ ~ \vdots~ ~\vdots$};   
\draw (8.3,-3.2) node[anchor=north west] {$\vdots~ ~\vdots~ ~ \vdots$};  
\draw (14,-3.2) node[anchor=north west] {$\vdots$};

\draw (10.4,3) node[anchor=north west] {$1$};

\draw (7.6,1.5) node[anchor=north west] {$\{2,3\}$};
\draw (11.8,1.5) node[anchor=north west] {$3$};

\draw (3.3,-1.4) node[anchor=north west] {$\{4,5\}$};
\draw (7.3,-1.4) node[anchor=north west] {$5$};
\draw (9.8,-1.4) node[anchor=north west] {$7$};
\draw (14.6,-1.4) node[anchor=north west] {$5$};
\end{tikzpicture}

 Note that the numbers that appears on either side of the edges is the elements that we remove from the vertex to obtain its son.


\section{Examples of finite trees}\label{S4}

Given $k\in \N$ we denote by \[\mathscr C(k)=\left\{S ~|~ S ~ \text{is an Coe-semigroup and} ~k\in S\right\}.\]

If $k$ is an even positive integer and $n\in \left\{x\in \N ~|~ x> k\right\}$, then $\big(\{k\}+\langle 2\rangle\big)\cup \left\{0, n,\rightarrow\right\}\in \mathscr C(k)$ and so, in this case, $\mathscr C(k)$ has infinite cardinality.
Our first aim in this section is to see  that $\mathscr C(k)$ has finite cardinality, when $k$ is an odd  positive integer.

If $S$ is a numerical semigroup, then $\N\backslash S$ is  a finite set and so we obtain the following result.

\begin{lemma}\label{19}
If $S$ is a numerical semigroup,  then $\left\{T ~|~ T~ \text{is a numerical semigroup and } ~S\subseteq T\right\}$ is a finite set.
\end{lemma}

\begin{proposition}\label{20}
If $k$ is an odd positive integer, then $\mathscr C(k)$ is a nonempty finite set.
\end{proposition}
\begin{proof}
Since $\N\in\mathscr C(k)$, then $\mathscr C(k)\neq \emptyset$. If $S\in \mathscr C(k)$, then we have that $\langle k-1,k,k+1\rangle\subseteq S$ and it is clear that $\langle k-1,k,k+1\rangle$ is  a numerical semigroup. Besides,  $\mathscr C(k)\subseteq \left\{T ~|~ T~ \text{is a numerical semigroup and } ~ \langle k-1,k,k+1\rangle \subseteq T\right\}$. By Lemma \ref {19}, we can deduce that $\mathscr C(k)$ is a finite set.
\end{proof}

The  following result is easy to prove.
\begin{lemma}\label{21}
 Let $k$ be a positive integer and  let $S\in\mathscr C(k)$ such that $S\neq \N$. Then $S\cup \left\{\F(S), \F(S)-1\right\}\in \mathscr C(k)$.
\end{lemma}

We define the graph $G\big(\mathscr C(k)\big)$ as the graph whose vertices are elements of $\mathscr C(k)$ and $(S, T)\in \mathscr C(k)\times \mathscr C(k)$  is an edge if $T=S\cup \left\{\F(S)-1,\F(S)\right\}$.  

Using the same argument of Section $3$, we have the following result.

\begin{theorem}\label{22}
 If $k$ is a positive integer, then the graph $G\big(\mathscr C(k)\big)$ is a tree with root equal to $\N$. Furthermore,                                         if $S\in \mathscr C(k)$ then the set of sons of $S$ in the  tree  $G\big(\mathscr C(k)\big)$  is equal to:
\begin{enumerate}
\item  $\left\{S\backslash\{x\} ~|~ x ~ \text{is an odd element in}~ \msg (S) ~\text{with}~x>\F(S)~ \text{and}~x\neq k \right\}$ if  $\left\{\F(S)+1,\F(S)+2\right\}\nsubseteq \msg\big(S\big)\backslash\{k\}$. 
\item  $\left\{S\backslash\{x\} ~|~ x ~ \text{is an odd element in}~ \msg (S) ~\text{with}~x>\F(S) ~\text{and}~x\neq k \right\}$ $\cup \left\{S\backslash \left\{\F(S)+1, \F(S)+2\right\}\right\}$ if $\left\{\F(S)+1,\F(S)+2\right\}\subseteq \msg\big(S\big)\backslash\{k\}$. 
\end{enumerate}
\end{theorem}

\begin{example}\label{23}
We are going to build the the tree   $G\big(\mathscr C(5)\big)$.

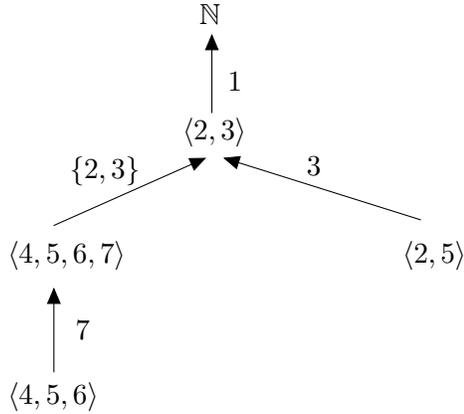
\begin{tikzpicture}[line cap=round,line join=round,>=triangle 45,xscale=0.75,yscale=0.75]
\clip(0,-5) rectangle (16,5);

\draw [<-] (10.3,3.5) -- (10.3,2.1);

\draw [<-] (10.2,1.3) -- (7.5,0.1);
\draw [<-] (10.5,1.3) -- (14,0.2);

\draw [<-] (7.5,-1) -- (7.5,-2.5);

\draw (9.9,4.2) node[anchor=north west] {$\N$};

\draw (9.6,2.2) node[anchor=north west] {$\langle 2,3\rangle$};

\draw (6.5,0) node[anchor=north west] {$\langle 4,5,6,7\rangle$};
\draw (13.5,0) node[anchor=north west] {$\langle 2,5\rangle$};

\draw (6.5,-2.5) node[anchor=north west] {$\langle 4,5,6\rangle$};

\draw (10.4,3) node[anchor=north west] {$1$};

\draw (7.6,1.5) node[anchor=north west] {$\{2,3\}$};
\draw (11.8,1.5) node[anchor=north west] {$3$};

\draw (7.7,-1.4) node[anchor=north west] {$7$};
\end{tikzpicture}
\end{example}

Given a positive integer $F$, denote by 
\[\mathscr C\big(Frob\leq F\big)=\left\{S ~|~ S ~ \text{is an Coe-semigroup with} ~ \F(S)\leq F\right\}.\]

Note that if $S$ is an element in $\mathscr C\big(Frob\leq F\big)$, then $\left\{F+1,\rightarrow\right\}\subseteq S$ and thus  $\mathscr C\big(Frob\leq F\big)$ is a finite set.

The  next result is easy to prove.
\begin{lemma}\label{24}
 Let $F$ be a positive integer and  let $S\in \mathscr C\big(Frob\leq F\big)$ such that $S\neq \N$. Then $S\cup \left\{\F(S), \F(S)-1\right\}\in \mathscr C\big(Frob\leq F\big)$.
\end{lemma}

Now, we define the graph $G\big(\mathscr C (Frob\leq F)\big)$ as follows:  $\mathscr C \big(Frob\leq F\big)$ is its set of vertices and $(S, T)\in \mathscr C \big(Frob\leq F\big)\times \mathscr C \big(Frob\leq F\big)$  is an edge if $T=S\cup \left\{\F(S)-1,\F(S)\right\}$.  

By using  again the same argument of Section $3$, we have the next result.

\begin{theorem}\label{25}
 If $F$ is a positive integer, then the graph $G\big(\mathscr C (Frob\leq F)\big)$ is a tree with root equal to $\N$. Furthermore,                                         if $S\in \mathscr C \big(Frob\leq F\big)$ then the set of sons of $S$ in the  tree  $G\big(\mathscr C (Frob\leq F)\big)$ is equal to:
\begin{enumerate}
\item   $\left\{S\backslash\{x\} ~|~ x ~ \text{is an odd element in}~ \msg (S) ~\text{with}~\F(S)<x\leq F \right\}$  if  $\left\{\F(S)+1,\F(S)+2\right\}\nsubseteq \left\{x\in \msg(S)~|~ x\leq F\right\}$. 
\item     $\left\{S\backslash\{x\} ~|~ x ~ \text{is an odd element in}~ \msg (S) ~\text{with}~\F(S)<x\leq F\right\}$ $ \cup \left\{S\backslash \left\{\F(S)+1, \F(S)+2\right\}\right\}$ if   $\left\{\F(S)+1,\F(S)+2\right\}\subseteq \left\{x\in \msg(S) ~|~ x\leq F\right\}$. 
\end{enumerate}
\end{theorem}

\begin{example}\label{26}
We are going to build the tree   $G\big(\mathscr C (Frob\leq 5)\big)$.

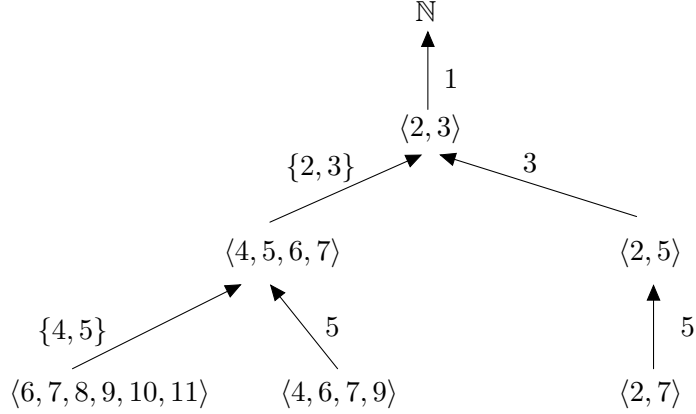
\begin{tikzpicture}[line cap=round,line join=round,>=triangle 45,xscale=0.75,yscale=0.75]
\clip(0,-5) rectangle (16,5);

\draw [<-] (10.3,3.5) -- (10.3,2.1);

\draw [<-] (10.2,1.3) -- (7.5,0.1);
\draw [<-] (10.5,1.3) -- (14,0.2);

\draw [<-] (14.3,-1) -- (14.3,-2.5);
\draw [<-] (7,-1) -- (4,-2.5);
\draw [<-] (7.5,-1) -- (8.7,-2.5);

\draw (9.9,4.2) node[anchor=north west] {$\N$};

\draw (9.6,2.2) node[anchor=north west] {$\langle 2,3\rangle$};

\draw (6.5,0) node[anchor=north west] {$\langle 4,5,6,7\rangle$};
\draw (13.5,0) node[anchor=north west] {$\langle 2,5\rangle$};

\draw (2.7,-2.5) node[anchor=north west] {$\langle 6,7,8,9,10,11\rangle$};
\draw (7.5,-2.5) node[anchor=north west] {$\langle 4,6,7,9\rangle$};
\draw (13.5,-2.5) node[anchor=north west] {$\langle 2,7\rangle$};

\draw (10.4,3) node[anchor=north west] {$1$};

\draw (7.6,1.5) node[anchor=north west] {$\{2,3\}$};
\draw (11.8,1.5) node[anchor=north west] {$3$};

\draw (3.2,-1.4) node[anchor=north west] {$\{4,5\}$};
\draw (8.3,-1.4) node[anchor=north west] {$5$};
\draw (14.6,-1.4) node[anchor=north west] {$5$};
\end{tikzpicture}
\end{example}

Given a positive integer $g$, denote by 
\[\mathscr C\big(gen\leq g\big)=\left\{S ~|~ S ~ \text{is an Coe-semigroup with} ~ \g(S)\leq g\right\}.\]
In \cite[Lemma $2.14$]{libro} it is shown that, if $S$ is a numerical semigroup then $\F(S)\leq 2\g(S)-1$. Therefore, we obtain that $\mathscr C\big(gen\leq g\big)\subseteq \mathscr C\big(Frob\leq 2g-1\big)$ and so $\mathscr C\big(gen\leq g\big)$ is a finite set.

The  next result is easy to prove.
\begin{lemma}\label{27}
 Let $g$ be a positive integer and  let $S\in \mathscr C\big(gen\leq g\big)$ such that $S\neq \N$. Then $S\cup \left\{\F(S), \F(S)-1\right\}\in \mathscr C\big(gen\leq g\big)$.
\end{lemma}

We define the graph $G\big(\mathscr C (gen\leq g)\big)$ as follows:  $\mathscr C \big(gen\leq g\big)$ is its set of vertices and $(S, T)\in \mathscr C \big(gen\leq g\big)\times \mathscr C \big(gen\leq g\big)$  is an edge if $T=S\cup \left\{\F(S)-1,\F(S)\right\}$.  

By using  again the same argument of Section $3$, we have the next result.

\begin{theorem}\label{28}
 If $g$ is a positive integer, then the graph $G\big(\mathscr C (gen\leq g)\big)$ is a tree with root equal to $\N$. Furthermore,                                         if $S\in \mathscr C \big(gen\leq g\big)$ then the set of sons of $S$ in the  tree  $G\big(\mathscr C (gen\leq g)\big)$ is equal to:
\begin{enumerate}
\item the set of sons of $S$ in the tree  $G\big(\mathscr C\big)$ if $\g(S)\leq g-2$ (see Theorem \ref{18}).
\item   $\left\{S\backslash\{x\} ~|~ x ~ \text{is an odd element in}~ \msg (S) ~\text{with}~x>\F(S) \right\}$ if  $\g(S)=g-1$.
\item  the empty set if $\g(S)=g$. 
\end{enumerate}
\end{theorem}

\begin{example}\label{29}
We are going to build the the tree   $G\big(\mathscr C (gen\leq 4)\big)$.

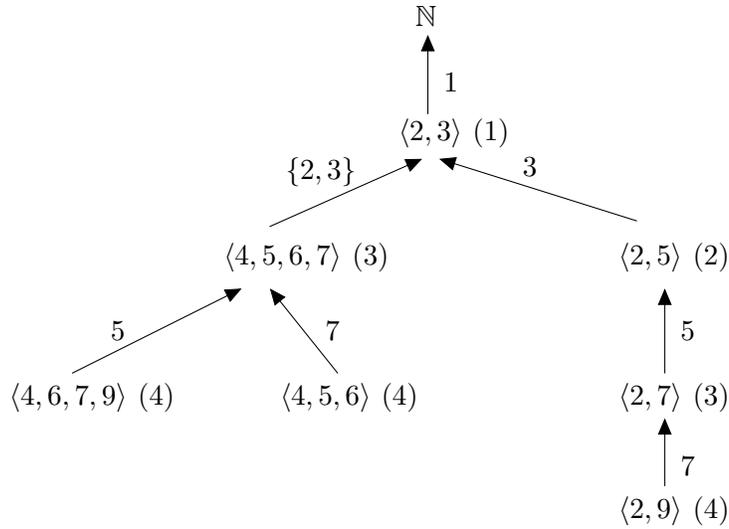
\begin{tikzpicture}[line cap=round,line join=round,>=triangle 45,xscale=0.75,yscale=0.75]
\clip(0,-5.5) rectangle (16,5);

\draw [<-] (10.3,3.5) -- (10.3,2.1);

\draw [<-] (10.2,1.3) -- (7.5,0.1);
\draw [<-] (10.5,1.3) -- (14,0.2);

\draw [<-] (7,-1) -- (4,-2.5);
\draw [<-] (7.5,-1) -- (8.7,-2.5);
\draw [<-] (14.5,-1) -- (14.5,-2.5);
\draw [<-] (14.5,-3.3) -- (14.5,-4.5);

\draw (9.9,4.2) node[anchor=north west] {$\N$};

\draw (9.6,2.2) node[anchor=north west] {$\langle 2,3\rangle$~(1)};

\draw (6.5,0) node[anchor=north west] {$\langle 4,5,6,7\rangle$~(3)};
\draw (13.5,0) node[anchor=north west] {$\langle 2,5\rangle$~(2)};

\draw (2.7,-2.5) node[anchor=north west] {$\langle 4, 6,7,9\rangle$~(4)};
\draw (7.5,-2.5) node[anchor=north west] {$\langle 4,5,6\rangle$~(4)};
\draw (13.5,-2.5) node[anchor=north west] {$\langle 2,7\rangle$~(3)};
\draw (13.5,-4.5) node[anchor=north west] {$\langle 2,9\rangle$~(4)};

\draw (10.4,3) node[anchor=north west] {$1$};

\draw (7.6,1.5) node[anchor=north west] {$\{2,3\}$};
\draw (11.8,1.5) node[anchor=north west] {$3$};

\draw (4.5,-1.4) node[anchor=north west] {$5$};
\draw (8.3,-1.4) node[anchor=north west] {$7$};
\draw (14.6,-1.4) node[anchor=north west] {$5$};
\draw (14.6,-3.8) node[anchor=north west] {$7$};
\end{tikzpicture}
\end{example}

 Observe that the numbers that appears on the right side of each of the semigroups are their genders.

\section{Coe-monoids}\label{S5}

It is well known that the intersection of finitely many numerical semigroups is a numerical semigroup. Hence, the next result is easy to prove.

\begin{proposition}\label{30}
The intersection of finitely  many Coe-semigroups is an Coe-semigroup.
\end{proposition}

Note that the previous result does not hold for infinite intersections. In fact, by using Proposition \ref{4}, we can deduce that $\langle 2,2k+1\rangle$ is an Coe-semigroup for all $k\in \N$ and $\bigcap_{k\in\N} \langle 2,2k+1\rangle=\langle 2\rangle$ is not a numerical semigroup.

Clearly, the intersection (finite or infinite) of numerical semigroups is a submonoid of $(\N,+)$. An Coe-monoid is a submonoid of $(\N,+)$ which can be expressed as intersection of Coe-semigroups. Besides, the intersection of Coe-monoids is an Coe-monoid.  In view of this, given $X$ a subset of $\N$ we can define the Coe-monoid generated by $X$ as the intersection of all Coe-monoids containing $X$, denoted by $Coe(X)$. Then, we have that $Coe(X)$ is the smallest Coe-monoid containing $X$.

The next result has immediate proof.

\begin{proposition}\label{31}
If $X$ is a subset of $\N$, then  $Coe(X$) is the intersection of all  Coe-semigroups that containing $X$.
\end{proposition}

If $M$ is an Coe-monoid and $X$ is a subset of $\N$ such that $M=Coe(X)$, then we say that $X$ is an coe-system of generators of $M$.

The following properties are a direct consequence of the definitions.

\begin{lemma}\label{32}
 Let $X$ and $Y$ be subsets of $\N$ and let $M$ be an Coe-monoid. Then the following conditions hold:
\begin{enumerate}
\item If $X\subseteq Y$ then $Coe(X)\subseteq Coe(Y)$,   
\item   $Coe(X)=Coe(\langle X\rangle)$,
\item   $Coe(M)=M$,
\item   $Coe(X\backslash\{0\})=Coe(X)$,
\item   $Coe(\emptyset)=\{0\}$.
\end{enumerate}
\end{lemma}

\begin{proposition}\label{33}
Let $M$ be an Coe-monoid such that $M\neq\{0\}$. Then there exists a  nonempty finite subset of $\N\backslash\{0\}$ such that $M=Coe(X)$.
\end{proposition}
\begin{proof}
We have that $\msg(M)$ is a nonempty finite subset $\N\backslash\{0\}$ such that $M=\langle \msg(M)\rangle$. By applying Lemma \ref{32}, we deduce that $M=Coe(\msg(M))$.
\end{proof}

As a consequence of previous proposition we have the following result.
\begin{corollary}\label{34}
The set formed by all Coe-monoids is equal to $\left\{Coe(X) ~|~ X~\text{is a  nonempty finite subset of}~ \N\backslash\{0\}\right\}\cup\left\{\{0\}\right\}$.
\end{corollary}

Given  $X$ a nonempty finite subset of $\N\backslash\{0\}$, our next aim in this section will be to show a procedure that allows us to compute $Coe(X)$.

\begin{theorem}\label{35}
 Let $M$ be a submonoid of $(\N,+)$ such that $M\neq\{0\}$. The following conditions are equivalent.
\begin{enumerate}
\item $M$ is an Coe-monoid.
\item $\left\{x-1, x+1\right\}\subseteq M$ for all odd element $x$ in $M$.
\item   $\left\{x-1, x+1\right\}\subseteq M$ for all odd element $x$ in $\msg(M)$.
\end{enumerate}
\end{theorem}

\begin{proof}
The equivalence between conditions 2) and 3)  is analogous to the proof of Proposition \ref{4}.

$1)~ \textit{implies}~ 2)$ If $M$ is an Coe-monoid, then there exists a family  $\{S_i\}_{i\in I}$ of Coe-semigroups such that $M=\bigcap_{i\in I}S_i$. If $x$ is an odd element in $M$, then $x\in S_i$ for all $i\in I$. As  $\left\{x-1, x+1\right\}\subseteq S_i$ for all $i\in I$, then we get that $\left\{x-1, x+1\right\}\subseteq \bigcap_{i\in I}S_i=M$.

$2)~ \textit{implies}~ 1)$  It is clear that $M_k=M\cup \left\{2k,\rightarrow\right\}$ is an Coe-semigroup for all $k\in\N$. Hence,   $M=\bigcap_{k\in \N}M_k$ is an Coe-monoid.
\end{proof}

\begin{corollary}\label{36}
 Let $M$ be a submonoid of $(\N,+)$. Then $M$ is  an Coe-monoid if and only if either $M\subseteq \langle 2\rangle$ or $M$ is an Coe-semigroup.
\end{corollary}
\begin{proof}
If $M$ is an Coe-monoid which does not contain odd elements, then $M\subseteq \langle 2\rangle$. On the other hand, if  $x$ is an odd element in $M$, then by Theorem \ref{35}, $\left\{x-1, x, x+1\right\}\subseteq M$. Therefore, we can conclude that $M$ is a numerical semigroup which is an Coe-semigroup.

Conversely, if $M\subseteq \langle 2\rangle$, then, by applying condition $2)$ of Theorem \ref{35}, we get that $M$ is an Coe-monoid. Furthermore, if $M$ is an Coe-semigroup, then it is an Coe-monoid.
\end{proof}

\begin{corollary}\label{37}
 If $X$ is a subset of $\N\backslash\{0\}$, then $Coe(X)$ is a submonoid of $(\N,+)$ generated by $X\cup\left\{x+1 ~|~ x ~ \text{is an odd element in}~ X\right\}\cup\left\{x-1 ~|~ x ~ \text{is an odd element in}~ X\right\}$.
\end{corollary}
\begin{proof}
Let $A=X\cup\left\{x+1 ~|~ x ~ \text{is an odd element in}~ X\right\}\cup\left\{x-1 ~|~ x ~ \text{is an odd element in}~ X\right\}$.
By condition $2)$ of Theorem \ref{35}, we get that $\langle A\rangle \subseteq Coe(X)$ and by condition $3)$ of Theorem \ref{35}, we obtain that $Coe(X)\subseteq A$. Whence,  $\langle A\rangle = Coe(X)$. 
\end{proof}

\begin{example}\label{38}
By using Corollary \ref{37}, we obtain that $Coe(\left\{4,7\right\})=\langle 4,6,7,8\rangle=\langle 4,6,7\rangle$.
\end{example}

 The following result is easy to prove.

\begin{corollary}\label{39}
 Let $X$ be a subset of $\N\backslash\{0\}$. Then, $Coe(X)$ is an Coe-semigroup if and only if $X$ contains at least one odd element.
\end{corollary}


\section{Coe-semigroups with maximal embedding dimension}\label{S6}

It is known that, if $S$ is a numerical semigroup, then its embedding dimension, $\e(S)$, is less than or equal to its multiplicity, $\m(S)$ (see \cite[Proposition $2.10$]{libro}. And $S$ has maximal embedding dimension (Med-semigroup, for short) whenever  $\e(S)=\m(S)$.
From \cite[Proposition $I.2.9$]{barucci} we can deduce the next result.

\begin{lemma}\label{40}
Let $S$ be a numerical semigroup. Then $S$ is an Med-semigroup if and only if  $\left\{s-\m(S) ~|~ s\in S\backslash \{0\}\right\}$ is a numerical semigroup.
\end{lemma}

\begin{proposition}\label{41}
Let $S$ be an Med-semigroup such that $S\neq \N$. Then $S$ is an Coe-semigroup if and only if $\m(S)$ is even and $T=\{-\m(S)\}+\big(S\backslash \{0\}\big)$ is an Coe-semigroup.
\end{proposition}

\begin{proof}
$Necessity$.  From Lemma \ref{40}, we have that $T$ is a numerical  semigroup and by Proposition \ref{1}, $\m(S)$ is even. Let $t$ be an odd element in $T$. Then there exists  $s\in S$ such that $t=s-\m(S)$ and  $s$ is odd. Since $S$ is an Coe-semigroup then $\left\{s-1,s+1\right\}\subseteq S$ and so $\left\{t-1,t+1\right\}=\left\{s-1-\m(S), s+1-\m(S)\right\}\subseteq T$. Hence, $T$ is an Coe-semigroup.

$Sufficiency$.  Let $s$ be an odd element in $S$. Then $s-\m(S)$ is an odd element in $T$. As $T$ is an Coe-semigroup then $\left\{s-1-\m(S),s+1-\m(S)\right\}\subseteq T$ and so  $\left\{s-1,s+1\right\}\subseteq S$. Consequently, $S$ is an Coe-semigroup.
\end{proof}

From Lemma \ref{40}, is easy to prove the following result.

\begin{lemma}\label{42}
If $S$ is a numerical semigroup and  $x\in S\backslash \{0\}$, then $S_x=\big(\{x\}+S\big)\cup\{0\}$ is an Med-semigroup with 
$\m(S_x)=x$. Moreover, every Med-semigroup is of this form.
\end{lemma}

\begin{proposition}\label{43}
Let $S$ be an Coe-semigroup and let $x$ be an even element in $S\backslash \{0\}$. Then $S_x=\big(\{x\}+S\big)\cup\{0\}$ is an Coe-semigroup with maximal embedding dimension. Moreover, every Coe-semigroup with maximal embedding dimension, distinct of $\N$, is of this form.
\end{proposition}
\begin{proof}
By using Lemma \ref{42}, we know that $S_x$ is an Med-semigroup with $\m\big(S_x)=x$. Clearly, that $S=\{-x\}+\big(S_x\backslash \{0\}\big)$ and by Proposition \ref{41} we get that $S_x$ is an Coe-semigroup.

Now, let $T$ be an Coe-semigroup with maximal embedding dimension, distinct from $\N$. By Proposition \ref{1}, we have that $\m(T)$ is even and, by Proposition \ref{41}, we deduce that $Q=\{-\m(T)\}+\big(T\backslash \{0\}\big)$ is an Coe-semigroup. Finally, $T=\big(\{\m(T)\}+Q\big)\cup\{0\}$   wherein $Q$ is an Coe-semigroup and $\m(T)$ is an even element in $Q$.
\end{proof}

Let $S$ be a numerical semigroup and let $n\in S\backslash\{0\}$. The Apéry set (named so in honour of \cite {apery}) of $n$ in $S$ is
\[\mathrm{Ap}(S,n)= \left\{ s \in S ~|~ s-n \not\in S\right\}.\]

\begin{lemma}\cite[Lemma 2.4]{libro}\label{44}
Le $S$ be a numerical semigroup and $n\in S\backslash\{0\}$. Then $\mathrm{Ap}(S,n) = \{0 = w(0), w(1), \ldots, w(n-1)\}$, where $w(i)$  is the least element in $S$  congruent with $i$ modulo $n$, for all $i\in\left\{0,\ldots,n-1\right\}$.
\end{lemma}

Observe that the above lemma in particular implies that the cardinality of Ap$(S, n)$ is $n$. From \cite{belga}, we can deduce the next result.

\begin{proposition}\label{45}
Le $S$ be a numerical semigroup, $x\in S\backslash\{0\}$ and $T=(\{x\}+S)\cup \{0\}$ . Then the following conditions hold:
\begin{enumerate}
\item $T$ is an MED-semigroup
\item $\m(T)=x$.
\item  $\F(T)=\F(S)+x$.
\item $\g(T)=\g(S)+x-1$.
\item  $\msg(T)=\mathrm{Ap}(S,x)+\{x\}$.
\end{enumerate}
\end{proposition}

\begin{example}\label{46}
From Example \ref{5} we have that $S=\langle 4,6,7\rangle=\left\{0,4,6,7,8,10,\rightarrow\right\}$ is an Coe-semigroup. Since $6$ is an even element in $S$, then by applying Propositions \ref{43} and \ref{45}, we obtain that $T=(\{6\}+S)\cup \{0\}$ is an Coe-semigroup with maximal embedding dimension, $\m(T)=6$, $\F(T)=9+6=15$, $\g(T)=5+6-1=10$ and 
$\msg(T)=\mathrm{Ap}(S,6)+\{6\}=\left\{0,4,7,8,11,15\right\}+\{6\}=\left\{6,10,13,14,17,21\right\}$.
\end{example}


\section{Coe-semigroups with an unique odd minimal generator}\label{S7}

Our first aim in this section is to prove Theorem \ref{49}, which can be used to construct whole set of Coe-semigroups with an unique odd minimal generator.

\begin{lemma}\label{47}
Let $S$ be a numerical semigroup and $\left\{s,s+1\right\}\subseteq S$. Then $T=2S\cup \big(\{2s+1\}+2S\big)$ is an Coe-semigroup. Moreover, $2s+1$ is the unique odd minimal generator in $T$.
\end{lemma}
\begin{proof}
The sum of two elements in $2S$ belongs to $2S$. The sum of  two elements in $\{2s+1\}+2S$ belongs to $2S$. The sum of an element in $2S$ and an element in $\{2s+1\}+2S$ belongs to $\{2s+1\}+2S$. As a result of the previous observations, we can conclude that $T$ is a numerical semigroup. Besides, if $\msg (S)=\left\{n_1,\ldots,n_p\right\}$, then $\left\{2n_1,\ldots,2n_p, 2s+1\right\}$ is a system of generators of $T$ and $2s+1$ is the unique odd minimal generator in $T$. Since $\left\{s,s+1\right\}\subseteq S$, then $\left\{2s,2s+2\right\}\subseteq T$ and so, by Proposition \ref{4}, we obtain that  $T$ is an Coe-semigroup.
\end{proof}

\begin{lemma}\label{48}
Let $T$ be an Coe-semigroup, with $T\neq \N$, and $x$ the unique odd minimal generator in $T$. Then there exists a numerical semigroup $S$ such that $T=2S\cup \big(\{x\}+2S\big)$ and $\left\{\frac{x-1}{2},\frac{x-1}{2}+1\right\}\subseteq S$.
\end{lemma}
\begin{proof}
If $A=\left\{a\in\msg (T) ~|~ a ~ \text{is even}\right\}$, then $\msg (T)=A\cup\{x\}$. Since $T$ is an Coe-semigroup, then $\left\{x-1,x+1\right\}\subseteq T$ and so  $\left\{x-1,x+1\right\}\subseteq \langle A\rangle$. As $\gcd\left\{x-1,x+1\right\}=2$, then  $\gcd\left\{A\right\}=2$ and therefore $S=\langle\left\{\frac{a}{2} ~|~ a\in A\right\}\rangle$ is a numerical semigroup.

Since $\left\{x-1,x+1\right\}\subseteq \langle A\rangle$, then $2x=(x-1)+(x+1)\in \langle A\rangle$  and we obtain that $T=\langle A\rangle\cup \big(\{x\}+\langle A\rangle\big)$. Besides, we have that $\langle A\rangle=2S$, then  $T=2S\cup \big(\{x\}+2S\big)$. Finally,  if $\left\{x-1,x+1\right\}\subseteq \langle A\rangle$, then  we get that $\left\{\frac{x-1}{2},\frac{x+1}{2}\right\}= \left\{\frac{x-1}{2},\frac{x-1}{2}+1\right\}\subseteq S$.
\end{proof}

As a consequence of  Lemmas \ref{47} and \ref{48}, we obtain the next result.

\begin{theorem}\label{49}
 If $S$ is a numerical semigroup and $\left\{s,s+1\right\}\subseteq S$, then $T=2S\cup \big(\{2s+1\}+2S\big)$ is an Coe-semigroup with an unique odd minimal generator. Moreover, every Coe-semigroup with an unique odd minimal generator is of this form.
\end{theorem}
\medskip

\begin{proposition}\label{50}
Let $S$ be a numerical semigroup, and $\left\{s,s+1\right\}\subseteq S$, then $T=2S\cup \big(\{2s+1\}+2S\big)$. Then the following conditions hold:
\begin{enumerate}
\item $\m(T)=2\m(S)$.
\item  $\F(T)=2\F(S)+(2s+1)$.
\item $\g(T)=2\g(S)+s$.
\item  $\msg(T)=(2\msg(S))\cup\{2s+1\}$.
\item $\e(T)=\e(S)+1$.
\end{enumerate}
\end{proposition}
\begin{proof}
\begin{enumerate}
\item  Trivial.
\item  It is sufficient to see that all even elements greater that $2\F(S)$ and all odd elements greater that $2s+1+2\F(S)$ belongs to $T$.
Since $2s+1+2\F(S)\not\in T$, then $\F(T)=2s+1+2\F(S)$
\item The cardinality of the set of the even elements not  in $T$ is equal to $\g(S)$. The set of the odd elements not in $T$ is equal to $\left\{2k+1 ~|~k\in\left\{0,1,\ldots, s-1\right\}\right\}\cup\left\{2s+1+2x ~|~x\in \N\backslash S\right\}$. Hence, we get that $\g(T)=2\g(S)+s$.
\item  If $\msg (S)=\left\{n_1,\ldots,n_p\right\}$, then $T=\langle 2n_1,\ldots,2n_p, 2s+1\rangle$. In fact, $2s+1\not\in \langle 2n_1,\ldots,2n_p\rangle$. In order to conclude the proof, it suffices to show that $2n_1\not\in \langle 2n_2,\ldots,2n_p, 2s+1\rangle$. Suppose that there exists $\left\{\lambda_1,\lambda_2, \ldots, \lambda_p\right\}\subseteq \N$ such that $2n_1=\lambda_1(2s+1)+\lambda_2(2n_2)+\cdots +\lambda_p(2n_p)$. Hence, we get  that $\lambda_1$ is even and $n_1=\frac{\lambda_1}{2}(2s+1)+\lambda_2 n_2+\cdots +\lambda_p n_p$, with $\left\{\frac{\lambda_1}{2},\lambda_2, \ldots, \lambda_p\right\}\subseteq \N$. Since $\msg (S)=\left\{n_1,\ldots,n_p\right\}$ and  $2s+1\in S$, then $n_1= 2s+1$. We have that $s=\frac{n_1-1}{2}$, and as $\left\{s,s+1\right\}\subseteq S$, then $\left\{\frac{n_1-1}{2},\frac{n_1+1}{2} \right\}\subseteq S\backslash\{0\}$ . Therefore, $\frac{n_1-1}{2}+\frac{n_1+1}{2}=n_1$ and so $n_1\not\in \msg (S)$, a contradiction.
\item  Follow directly from the $(4)$.
\end{enumerate}
\end{proof}

\begin{example}\label{51}
If $S=\langle 5,7,9\rangle$ then $\m(S)=5$,  $\F(S)=13$, $\g(S)=8$, $\msg(S)=\left\{5,7,9\right\}$ and $\e(S)=3$.
Since $\left\{14,15\right\}\subseteq S$, then by Theorem \ref{49}, $T=2S\cup \big(\{29\}+2S\big)$ is an Coe-semigroup. By Proposition \ref{50}, we obtain that $m(T)=10$,  $\F(T)=55$, $\g(T)=30$, $\msg(T)=\left\{10,14,18,29\right\}$ and $\e(T)=4$.
\end{example}

Let $S$ a numerical semigroup. We say that $s$ is a small element in $S$ if $s<\F(S)$. Denote by $N(S)$  the set of all small elements in $S$ and by $n(S)$ its  cardinality. Note that $\F(S)+1=\g(S)+\n(S)$.

In 1978, H. S. Wilf (see \cite{wilf})  conjectured upper bound for $\g(S)$, namely,   $\g(S)\leq (\e(S)-1) n(S)$. Nowadays, the Wilf's conjecture remains unanswered, but for some families of numerical semigroups this conjecture is known to be true (see for example, \cite{delgado}).

\begin{proposition}\label{52}
Let $S$ be a numerical semigroup, $\left\{s,s+1\right\}\subseteq S$ and $T=2S\cup \big(\{2s+1\}+2S\big)$. If $S$ verifies the Wilf's conjecture, then $T$ also verifies the same conjecture. 
\end{proposition}
\begin{proof}
As $\n(T)=\F(T)+1-\g(T)$, then by Proposition \ref{50}, we have that $\n(T)=2\F(S)+(2s+1)+1-(2\g(S)+s)=2(\F(S)+1-\g(S))+s=2\n(S)+s$. We have that $T$ verifies the Wilf's conjecture if \\ 
\[\g(T)\leq (\e(T)-1) n(T) \Longleftrightarrow\]  \[ 2\g(S)+s \leq \e(S)(2\n(S)+s) \Longleftrightarrow \]  \[\g(S)\leq (\e(S)-1)\n(S)+\n(S)+\frac{(\e(S)-1)s}{2}.\]
Since  $S$ verifies the Wilf's conjecture, then $\g(S)\leq (\e(S)-1) n(S)$, $\n(S)\geq 0$ and $\frac{(\e(S)-1)s}{2}\geq 0$ and therefore $T$ also verifies the same conjecture. 
\end{proof}

\begin{example}\label{53}
If $S$ is a numerical semigroup with $\e(S)=3$, then by  \cite[Theorem $2.11$]{doobs}, we have that $S$ verifies the Wilf's conjecture. If $\left\{s,s+1\right\}\subseteq S$, then by Proposition \ref{50}, we obtain that $T=2S\cup \big(\{2s+1\}+2S\big)$ is an Coe-semigroup which verifies he Wilf's conjecture.   
\end{example}

We finish this section by studying the class of Coe-semigroups with embedding dimension $1$, $2$ and $3$. Clearly, the numerical semigroups, $\N$ and $\langle 2,2k+1\rangle$ with $k\in \N\backslash\{0\}$, are all Coe-semigroups with embedding dimension $1$ and $2$.

\begin{lemma}\label{54}
If $S$ is an Coe-semigroup with $\e(S)=3$, then $S$ has an  unique odd minimal generator.
\end{lemma}
\begin{proof}
 If $\msg (S)=\left\{n_1<n_2<n_3\right\}$, then by Proposition \ref{1}, we get that $n_1$ is an even integer. If $n_2$ and $n_3$ are two odd integers, as $\left\{n_2-1,n_2+1\right\}\subseteq T$, then $\left\{n_2-1,n_2+1\right\}\subseteq \langle n_1\rangle$ and so $n_1=2$. Therefore, $\m(S)=2$ and $\e(S)\leq \m(S)=2$ a contradiction.
\end{proof}

\begin{proposition}\label{55}
Let $a$ and $b$ be positive integers  such that $2\leq a<b$, $\gcd\{a,b\}=1$ and $\left\{s,s+1\right\}\subseteq \langle a,b\rangle$. Then  $T=\langle 2a,2b, 2s+1\rangle$ is  an Coe-semigroup with embedding dimension $3$.  Moreover, every Coe-semigroup with embedding  dimension $3$ is of this form.
\end{proposition}

\begin{proof}
As an immediate consequence of   Theorem \ref{49}, Proposition \ref{50} and Lemma \ref{54}.
\end{proof}

The next result appears in  \cite{Sylvester}.

\begin{lemma}\label{56}
Let $a$ and $b$ be positive integers  such that $2\leq a<b$, $\gcd\{a,b\}=1$. Then:
\begin{enumerate}
  \item $\F(\langle a,b\rangle)=ab-a-b$.
 \item  $\g(\langle a,b\rangle)=\frac{(a-1)(b-1)}{2}$.
\end{enumerate}
\end{lemma}

By applying Propositions \ref{50}, \ref{55} and Lemma \ref{56}, we obtain the next result.

\begin{proposition}\label{57}
If $S$ is an Coe-semigroup, $\msg(S)=\left\{n_1, n_2, n_3\right\}$ and $n_3$ an odd integer, then:
\begin{enumerate}
  \item $\F(S)=n_3+\frac{n_1n_2}{2}-n_1-n_2$.
 \item  $\g(S)=\frac{n_3-1}{2}+(\frac{n_1}{2}-1)(\frac{n_2}{2}-1)$.
\end{enumerate}
\end{proposition}

From   \cite[Lemma $2.14$]{libro} we know that if $S$ is a numerical semigroup, then $\g(S)\geq \frac{\F(S)+1}{2}$. Following the terminology introduced in \cite{kunz} a numerical semigroup is symmetric if  $\g(S)= \frac{\F(S)+1}{2}$.

By Proposition \ref{57}, we have the following result.

\begin{corollary}\label{58}
If $S$ is an Coe-semigroup with $\e(S)=3$, then $S$ is a symmetric numerical semigroup.
\end{corollary}



\end{document}